\documentclass[10pt,reqno]{amsart}
\usepackage{amsmath}
\usepackage{amssymb}
\usepackage{amsthm}

\newlength{\defbaselineskip}
\newcommand{\setlinespacing}[1]%
           {\setlength{\baselineskip}{#1 \defbaselineskip}}

\numberwithin{equation}{section}

\newtheorem{thm}{Theorem}[section]

\newtheorem{lem}[thm]{Lemma}

\theoremstyle{definition}

\theoremstyle{remark}
\newtheorem{rem}[thm]{Remark}
\numberwithin{equation}{section}

\begin{document}

\author{Ihyeok Seo}

\title[Unique continuation for Schr\"odinger operators]
{On unique continuation for Schr\"odinger operators of fractional and higher orders}

\address{School of Mathematics, Korea Institute for Advanced Study, Seoul 130-722, Republic of Korea}

\email{ihseo@kias.re.kr}

% --------------------------------------------------------------------------------------------------------------



% --------------------------------------------------------------------------------------------------------------

\begin{abstract}
In this note we study the property of unique continuation for solutions
of $|(-\Delta)^{\alpha/2}u|\leq|Vu|$, where
$V$ is in a function class of potentials including $\bigcup_{p>n/\alpha}L^p(\mathbb{R}^n)$ for $n-1\leq\alpha<n$.
In particular, when $n=2$, our result gives a unique continuation theorem
for the fractional ($1<\alpha<2$) Schr\"odinger operator $(-\Delta)^{\alpha/2}+V(x)$ in the full range of $\alpha$ values.
\end{abstract}

\maketitle

% --------------------------------------------------------------------------------------------------------------

\section{Introduction}

As it is well-known, analytic functions that are representable by power series
have the property of unique continuation.
This means that they cannot vanish in any non-empty open set without being identically zero.
Note that a solution $u$ of the Cauchy-Riemann operator $\overline{\partial}$ in $\mathbb{R}^2$
(i.e., $\overline{\partial}u=0$) has the property since it is complex analytic.
The same result holds for the Laplace operator $\Delta$ in $\mathbb{R}^n$
since its solutions are harmonic functions that are still real analytic.
So it would be desirable to obtain such property for partial differential operators
whose solutions are not necessarily analytic, or even smooth.

Historically, the first such result is due to Carleman ~\cite{C}, who showed the property
for the Schr\"odinger operator $-\Delta+V(x)$ in $\mathbb{R}^2$
if $V\in L_{\textrm{loc}}^\infty(\mathbb{R}^2)$.
This was extended to higher dimensions $n\geq3$ by M\"{u}ller~\cite{M}.
Since then, a great deal of work was devoted to the case $V\in L_{\textrm{loc}}^p(\mathbb{R}^n)$, $p<\infty$.
This is because the potentials $V$ that arise in quantum physics need not be locally bounded
and more importantly it can be applied to the problem of absence of positive eigenvalues
of the Schr\"odinger operator.
Among others, Jerison and Kenig~\cite{JK} proved
the unique continuation for more general differential inequalities of the form $|\Delta u|\leq|Vu|$
when $V\in L_{\textrm{loc}}^{n/2}(\mathbb{R}^n)$ if $n>2$,
and $V\in \bigcup_{p>1}L_{\textrm{loc}}^p(\mathbb{R}^2)$ if $n=2$.
This result later turns out to be optimal in the context of $L^p$ potentials $V$ (\cite{KN,KT}),
and was extended to the higher orders
$|\Delta^m u|\leq|Vu|$ ($m\in\mathbb{N}$) when $V\in L_{\textrm{loc}}^{n/2m}(\mathbb{R}^n)$ if $n>2m$ (\cite{L}),
and $V\in \bigcup_{p>1}L_{\textrm{loc}}^p(\mathbb{R}^n)$ if $n=2m$ (\cite{Se}).

In this note we are concerned with more general cases
of fractional ($0<\alpha<2$) and higher ($\alpha>2$) orders:
\begin{equation}\label{FH}
|(-\Delta)^{\alpha/2}u|\leq|Vu|,
\end{equation}
where $(-\Delta)^{\alpha/2}$ is defined by means of the Fourier transform
$\mathcal{F}f$ (=\,$\widehat{f}$\,):
$$\mathcal{F}[(-\Delta)^{\alpha/2}f](\xi)=|\xi|^\alpha\widehat{f}(\xi).$$
To the best of our knowledge, all the known results on unique continuation for~\eqref{FH}
deal only with the case of even integers $\alpha=2m$, $m\in\mathbb{N}$.
When it comes to the other cases of $\alpha$, the difficulty comes from the fact that
$(-\Delta)^{\alpha/2}$ is a nonlocal operator
which means that $(-\Delta)^{\alpha/2}f(x)$ depends not just on $f(y)$ for $y$ near $x$ but on $f(y)$ for all $y$.
Moreover, it does not satisfy Leibnitz's rule of differentiation, in general.
In order to get around these difficulties,
we consider a function class $\mathcal{K}_\alpha$, $0<\alpha<n$, of potentials $V$ defined by
\begin{equation}\label{poten}
V\in\mathcal{K}_\alpha\quad\Leftrightarrow\quad
\lim_{r\rightarrow0}\sup_{x\in\mathbb{R}^n}\int_{|x-y|<r}\frac{|V(y)|}{|x-y|^{n-\alpha}}dy=0.
\end{equation}
The case $\alpha=2$ is just the usual Kato class introduced by Kato~\cite{K}
to study the self-adjointness of the Schr\"odinger operator.
Since then, it has played an important role in the study of many other properties of the operator (cf. \cite{Si}).
In order to extend these studies to $(-\Delta)^{\alpha/2}+V(x)$,
the class ~\eqref{poten} was introduced and used by several authors.
(See ~\cite{ZY} and the references therein.)
By making use of ~\eqref{poten}, we obtain here unique continuation results
for ~\eqref{FH} with $n-1\leq\alpha<n$.

Before stating our result precisely,
it should be emphasized that there are physical interests in the case $1<\alpha<2$
as well as the case $\alpha=2$.
Recently, following the path integral approach (\cite{FH}) to quantum mechanics,
Laskin~\cite{La,La2,La3} generalized the Feynman path integral to the L\'{e}vy one.
This generalization leads to fractional quantum mechanics governed by the fractional Schr\"odinger equation
$i\partial_t\psi=((-\Delta)^{\alpha/2}+V(x))\psi$, where $1<\alpha<2$.
The usual quantum mechanics corresponds to the case $\alpha=2$.
In particular, when $n=2$, Theorem~\ref{thm} below gives a unique continuation result for
the fractional Schr\"odinger operator $(-\Delta)^{\alpha/2}+V(x)$
in the full range of $\alpha$ values.

\begin{thm}\label{thm}
Let $n\geq2$ and $n-1\leq\alpha<n$.
Assume that $u$ is a nonzero solution of~\eqref{FH} such that
\begin{equation}\label{sol}
u\in L^1(\mathbb{R}^n)\quad\text{and}\quad(-\Delta)^{\alpha/2}u\in L^1(\mathbb{R}^n).
\end{equation}
Then it cannot vanish in any non-empty open set of $\mathbb{R}^n$ if $V$ is in the class $\mathcal{K}_\alpha$.
\end{thm}

\begin{rem}
It is currently conjectured\footnote{The conjecture was first formulated by B. Simon in his survey paper on Schr\"odinger semigroups~\cite{Si}.}
that ~\eqref{FH} for $\alpha=2$ has the unique continuation whenever $V\in\mathcal{K}_2$.
When $n=3$, it is known to be true (\cite{S}). In higher dimensions this conjecture is verified 
for $\Delta u=Vu$ with radial potentials (\cite{FGL}).
More generally, it is natural to ask whether the theorem can hold for $0<\alpha<n-1$.  
It should be also noted that the theorem can be applied to the stationary equation 
$$((-\Delta)^{\alpha/2}+V(x))u=Eu,$$
and the same result holds for $E\in\mathbb{C}$ by noting that $(-\Delta)^{\alpha/2}u=(E-V(x))u$ and 
the condition~\eqref{poten} is trivially satisfied for the constant $E$.
\end{rem}

Note that the class $\mathcal{K}_\alpha$ has the property that
$\bigcup_{p>n/\alpha}L^p(\mathbb{R}^n)\subset\mathcal{K}_\alpha\subset L_{\textrm{loc}}^1(\mathbb{R}^n)$.
In fact, if $V\in L^p(\mathbb{R}^n)$, we see that
\begin{equation}\label{pr1}
\sup_{x\in\mathbb{R}^n}\int_{|x-y|<r}|x-y|^{-(n-\alpha)}|V(y)|dy
\leq C\bigg(\int_{|y|<r}|y|^{-(n-\alpha)p^\prime}dy\bigg)^{1/p^\prime}
\end{equation}
by H\"{o}lder's inequality.
Using polar coordinates, if $p>n/\alpha$,
one can see that the right-hand side of ~\eqref{pr1} tends to $0$ as $r\rightarrow0$.
So it follows that $V\in\mathcal{K}_\alpha$.
On the other hand, if $V\in\mathcal{K}_\alpha$, there is $0<r_0<1$ so that
the left-hand side of ~\eqref{pr1} is less than $1$.
Hence we get
$$\sup_{x\in\mathbb{R}^n}\int_{|x-y|<r_0}|V(y)|dy\leq 1$$
since $|x|^{-(n-\alpha)}\geq1$ for $|x|<r_0$.
This implies that $V\in L_{\textrm{loc}}^1(\mathbb{R}^n)$.

\medskip

Throughout this paper, the letter $C$ stands for positive constants possibly different at each occurrence.

% --------------------------------------------------------------------------------------------------------------

\section{Preliminary lemmas}

In this section we present some preliminary lemmas which will be used
for the proof of Theorem~\ref{thm}.

\begin{lem}\label{lemma}
Let $\phi_\alpha(y)=|y|^{-(n-\alpha)}$ for $0<\alpha<n$.
Then we have for $x\in\mathbb{R}^n$ and $N\geq1$,
\begin{equation}\label{7}
u(x)=C\int_{\mathbb{R}^n}
\big[\phi_\alpha(x-y)-\sum_{k=0}^{N-1}\frac{(x\cdot\nabla)^k}{k!}\phi_\alpha(-y)\big](-\Delta)^{\alpha/2}u(y)dy
\end{equation}
if $u$ satisfies~\eqref{sol} and has a compact support in $\mathbb{R}^n\setminus\{0\}$.
\end{lem}

\begin{proof}
It is enough to show that ~\eqref{7} holds for $u\in C_0^\infty(\mathbb{R}^n\setminus\{0\})$.
The remaining follows from this and a standard limiting argument involving a $C_0^\infty$ approximate identity.

First we claim that
\begin{equation}\label{1}
u(x)=C\int_{\mathbb{R}^n}\phi_\alpha(x-y)(-\Delta)^{\alpha/2}u(y)dy.
\end{equation}
Indeed, using the well-known fact (cf.~\cite{W}, p.23) that
$\widehat{|x|^{-\alpha}}(\xi)=C|\xi|^{-(n-\alpha)}$ in the sense of distributional Fourier transforms,
we see that
\begin{align*}
u(x)=\int_{\mathbb{R}^n}e^{2\pi ix\cdot\xi}\widehat{u}(\xi)d\xi
&=\int_{\mathbb{R}^n}|\xi|^{-\alpha}\mathcal{F}\big[(-\Delta)^{\alpha/2}u(\cdot+x)\big](\xi)d\xi\\
&=C\int_{\mathbb{R}^n}|y|^{-(n-\alpha)}(-\Delta)^{\alpha/2}u(y+x)dy,
\end{align*}
which gives ~\eqref{1}.

Now, we note that the $(N-1)^{th}$ degree Taylor polynomial of $u$ at $0$ must be zero,
since $u$ vanishes near $x=0$.
That is to say,
$$\sum_{|\beta|\leq N-1}\frac{D^\beta u(0)}{\beta!}x^\beta\equiv0,$$
where $\beta$ is the usual multiindex notation.
By ~\eqref{1}, this can be rewritten as
\begin{equation}\label{sdfs2}
C\int_{\mathbb{R}^n}\sum_{|\beta|\leq N-1}\frac{D^\beta \phi_\alpha(-y)}{\beta!}x^\beta(-\Delta)^{\alpha/2}u(y)dy
\equiv0.
\end{equation}
Then, we subtract ~\eqref{sdfs2} from both sides of ~\eqref{1} to conclude
\begin{equation*}
u(x)=C\int_{\mathbb{R}^n}\big[\phi_\alpha(x-y)-\sum_{|\beta|\leq N-1}\frac{D^\beta \phi_\alpha(-y)}{\beta!}x^\beta\big](-\Delta)^{\alpha/2}u(y)dy
\end{equation*}
which is same as ~\eqref{7}.
\end{proof}

Nextly, we recall from ~\cite{S} the following estimate on Taylor polynomial approximations to
$|x|^{-\beta}$, $\beta>0$.

\begin{lem}\label{lem}
Let $\psi_\beta(x)=|x|^{-\beta}$ for $0<\beta\leq1$.
Then one has
\begin{equation}\label{Taylor}
\Big|\psi_\beta(x-y)-\sum_{k=0}^{N-1}\frac{(x\cdot\nabla)^k}{k!}\psi_\beta(-y)\Big|\leq
C\Big(\frac{|x|}{|y|}\Big)^N\psi_\beta(x-y)
\end{equation}
for $x,y\in\mathbb{R}^n$ and $N\geq1$.
Moreover, this estimate is not valid if $\beta>1$.
\end{lem}

% --------------------------------------------------------------------------------------------------------------

\section{Proof of Theorem~\ref{thm}}\label{Proof}

Without loss of generality, we need to show that $u$ must be identically zero
if $u$ vanishes in a sufficiently small neighborhood of $0\in\mathbb{R}^n$.

By using ~\eqref{FH}, ~\eqref{7} and ~\eqref{Taylor} with $\beta=n-\alpha$, if $n-1\leq\alpha<n$,
we see that
\begin{equation}\label{121}
|u(x)|\leq C\int_{\mathbb{R}^n}\big(\frac{|x|}{|y|}\big)^N\phi_\alpha(x-y)|V(y)u(y)|dy.
\end{equation}
Let $f(y)=|y|^{-N}|V(y)u(y)|$.
Since $u$ vanishes near the origin, it follows that
\begin{equation}\label{15789}
\|f\|_{L^1}=\int_{\mathbb{R}^n}|y|^{-N}|V(y)u(y)|dy<\infty.
\end{equation}
Here, we also used the fact (see Theorem 2.2 in~\cite{ZY}) that the condition ~\eqref{sol} implies
$Vu\in L^1$ if $V\in \mathcal{K}_\alpha$.
Hence, from~\eqref{121} we get
\begin{align}\label{lme}
\int_{|x|<r}|V(x)||x|^{-N}|u(x)|dx
\nonumber&\leq C\int_{\mathbb{R}^n}\bigg(\int_{|x|<r}\phi_\alpha(x-y)|V(x)|dx\bigg)f(y)dy\\
&\leq C\bigg(\sup_{y\in\mathbb{R}^n}\int_{|x|<r}\phi_\alpha(x-y)|V(x)|dx\bigg)\|f\|_{L^1}.
\end{align}

Now, we set
$$\eta(r)=\sup_{y\in\mathbb{R}^n}\int_{|x|<r}\phi_\alpha(x-y)|V(x)|dx.$$
Then the condition~\eqref{poten} implies $\lim_{r\rightarrow0}\eta(r)=0$.
In fact, we note that
$$\sup_{|y|<2r}\int_{|x|<r}\phi_\alpha(x-y)|V(x)|dx
\leq\sup_{y\in\mathbb{R}^n}\int_{|x-y|<4r}\phi_\alpha(x-y)|V(x)|dx$$
and
\begin{align*}
\sup_{|y|\geq2r}\int_{|x|<r}\phi_\alpha(x-y)|V(x)|dx
&\leq Cr^{\alpha-n}\int_{|x|<r}|V(x)|dx\\
&\leq Cr^{\alpha-n}r^{n-\alpha}\int_{|x|<r}\frac{|V(x)|}{|x|^{n-\alpha}}dx.
\end{align*}
Then it follows from ~\eqref{poten} that
$$\lim_{r\rightarrow0}\eta(r)\leq C\lim_{r\rightarrow0}\sup_{y\in\mathbb{R}^n}\int_{|x-y|<4r}\frac{|V(x)|}{|x-y|^{n-\alpha}}dx=0.$$

Hence, if we choose $r_0>0$ small enough, we see from ~\eqref{lme} that
\begin{equation*}
\int_{|x|<r_0}|V(x)||x|^{-N}|u(x)|dx\leq\frac12\|f\|_{L^1}.
\end{equation*}
Combining this with ~\eqref{15789}, we get
\begin{equation*}
\int_{|x|<r_0}|V(x)||x|^{-N}|u(x)|dx\leq\int_{|y|\geq r_0}|y|^{-N}|V(y)u(y)|dy,
\end{equation*}
so that
\begin{equation}\label{dks2}
\int_{|x|<r_0}|V(x)|\big(\frac{r_0}{|x|}\big)^N|u(x)|dx\leq\|Vu\|_{L^1}<\infty.
\end{equation}
Here we may assume that $|V|\geq1$,
since $|V|+1$ also satisfies ~\eqref{FH} and $\lim_{r\rightarrow0}\eta(r)=0$.
Indeed, to show the second one for $|V|+1$, we only need to show
\begin{equation}\label{357}
\lim_{r\rightarrow0}\sup_{y\in\mathbb{R}^n}\int_{|x|<r}\phi_\alpha(x-y)dx=0.
\end{equation}
Since $\phi_\alpha(x)=|x|^{-(n-\alpha)}$, it follows that
\begin{align*}
\sup_{y\in\mathbb{R}^n}\int_{|x|<r}\phi_\alpha(x-y)dx
&\leq\sup_{|y|<2r}\int_{|x|<r}|x-y|^{-(n-\alpha)}dx
+\sup_{|y|\geq2r}\int_{|x|<r}r^{-(n-\alpha)}dx\\
&\leq\sup_{y\in\mathbb{R}^n}\int_{|x-y|<4r}|x-y|^{-(n-\alpha)}dx+Cr^\alpha\\
&\leq Cr^\alpha.
\end{align*}
This gives ~\eqref{357}.

Therefore, from ~\eqref{dks2} we see that
$$\int_{|x|<\frac{r_0}2}2^N|u(x)|dx\leq
\int_{|x|<r_0}|V(x)|\big(\frac{r_0}{|x|}\big)^N|u(x)|dx<\infty.$$
By letting $N\rightarrow\infty$,
we get that $u$ vanishes in $\{|x|<\frac{r_0}2\}$.
Then, using a standard connectedness argument, we can conclude that $u$ must be identically zero in $\mathbb{R}^n$.
This completes the proof.

% --------------------------------------------------------------------------------------------------------------

\bibliographystyle{plain}

% --------------------------------------------------------------------------------------------------------------

\end{document}